\documentclass[11pt]{amsart}

\pdfoutput=1
\usepackage{amssymb}
\usepackage{graphicx}
\usepackage{amsthm}
\usepackage{amscd}
\usepackage{mathtools}
\usepackage{amssymb}
\usepackage{amsmath}
\usepackage{graphicx}
\usepackage{pict2e}
\usepackage[font=footnotesize]{caption}
\usepackage{graphicx}
\usepackage{tikz}
\usetikzlibrary{chains}

\tikzset{node distance=2em, ch/.style={circle,draw,on chain,inner sep=2pt},chj/.style={ch,join},every path/.style={shorten >=4pt,shorten <=4pt},line width=1pt,baseline=-1ex}

\usepackage{graphicx}

\usepackage{array}
\usepackage{verbatim}
\usepackage{color}
\usepackage[utf8]{inputenc}

\newtheorem{corollary}{Corollary}
\newtheorem*{corollary*}{Korollar}
\newtheorem*{notation*}{Notation}

\newtheorem*{dank*}{Danksagung}
\newtheorem{lemma}{Lemma}
\newtheorem*{example*}{Example}
\newtheorem{theorem}{Theorem}

\theoremstyle{remark}
\newtheorem{remark}{Remark}
\theoremstyle{theorema}

\newtheorem{theoremb}{Conjecture}

\theoremstyle{definition}
\newtheorem{definition}{Definition}

\DeclareMathOperator{\codim}{codim}

\title[On simplicial arrangements in $\mathbb{P}^3(\mathbb{R})$ with splitting polynomial]
{On simplicial arrangements in $\mathbb{P}^3(\mathbb{R})$ with splitting polynomial}

\author{David~Geis}
\email{davidgeis@web.de}

\begin{document}
\begin{abstract}
In this paper, we study simplicial hyperplane arrangements in real projective $3$-space. We give a necessary condition for the characteristic polynomial to have only real roots, valid also for non-simplicial arrangements. As application, we obtain combinatorial inequalities which are satisfied for arrangements with splitting polynomial. This allows us to prove that there are only finitely many different isomorphism classes of simply laced simplicial arrangements whose characteristic polynomials split over $\mathbb{R}$. We also provide an updated version of a catalogue published by Gr\"unbaum and Shephard and review some conjectures of theirs.
\end{abstract}


\maketitle

\begin{section}{Introduction}
In this note, we are interested in simplicial hyperplane arrangements in $\mathbb{P}^d(\mathbb{R})$. An arrangement is called simplicial if it divides the ambient space into open simplicial cones. The most prominent examples are probably provided by the reflection arrangements associated to finite real reflection groups. For $d=2$, there exist two infinite series of arrangements (see \cite{Gruenbaum}), but for $d \geq 3$, it is widely believed that there are only finitely many different isomorphism classes of \textit{irreducible} simplicial hyperplane arrangements in $\mathbb{P}^d(\mathbb{R})$ (the term ``irreducible'' will be explained in the following section). Moreover, for $d \in \lbrace 2,3 \rbrace$, Gr\"unbaum and Shephard gave catalogues in the papers \cite{Gruenbaum} and \cite{Gruenbaum_giessen}, which contained all such examples that were known at the time. In the papers \cite{Cuntz27}, \cite{Cuntz3space} and \cite{Cuntz4space}, several additional examples were discovered; in fact, for each $d \geq 2$ a complete subclass (so called \textit{crystallographic} arrangements) of simplicial arrangements in $\mathbb{P}^d(\mathbb{R})$ was classified. For fixed $d \geq 2$, each subclass is finite. In the paper \cite{cuntz_muecksch}, another classification result was given: for $d \geq 2$, all \textit{supersolvable} simplicial arrangements in $\mathbb{P}^d(\mathbb{R})$ were determined. Again, for fixed $d \geq 3$(!), each subclass is finite. Thus, there exists quite some evidence for the above mentioned belief. The aim of this paper is to provide further evidence: we fix $d:=3$ and show that there are only finitely many different isomorphism classes of \textit{simply laced} simplicial arrangements in $\mathbb{P}^3(\mathbb{R})$ with splitting characteristic polynomial (see Theorem \ref{simply laced thm}). This is achieved by providing inequalities involving the $h$-vector, $t$-vector and $f$-vector of an arrangement (see  for instance Theorem \ref{real roots rank 4}). We also review another conjecture made in \cite{Gruenbaum_giessen} (see Conjecture 2) and give an updated catalogue of simplicial arrangements in $\mathbb{P}^3(\mathbb{R})$. \\

The paper is organized as follows: in Section 2, we will first recall all needed definitions and concepts in the general case. We then fix $d=3$ and recall two conjectures made in \cite{Gruenbaum_giessen}. In the following Section 3, we formulate and prove our main results. Section 4 contains an updated catalogue of simplicial arrangements in $\mathbb{P}^3(\mathbb{R})$. Finally, the appendix contains normal vectors for two simplicial $3$-arrangements, which seem to be hard to find explicitly in the literature.   
\end{section}

\begin{section}{Definitions and setup}
In this section, we first introduce \textit{real simplicial projective hyperplane arrangements}, i.e. arrangements of hyperplanes in $\mathbb{P}^{d}(\mathbb{R})$ which induce a decomposition of the ambient space into simplicial cones (see Definition 1). After that, we shift focus towards the case $d=3$, where we will be especially interested in arrangements whose characteristic polynomials split over $\mathbb{R}$ (for instance \textit{free} arrangements); we also discuss some conjectures by Grünbaum and Shephard, before moving on to the next section. \\

We begin by introducing the objects of interest: hyperplane arrangements in projective spaces and some basic associated combinatorial concepts.
\begin{definition}
i) Let $d \geq 2$ be an integer. Write $\mathbb{P}^{d}:=\mathbb{P}^{d}(\mathbb{R})$ and let $\pi:\mathbb{R}^{d+1}\setminus \lbrace 0 \rbrace \longrightarrow \mathbb{P}^{d}$ be the natural map. If $H=\pi(H^\prime \setminus \lbrace 0 \rbrace)$ for some $H^\prime \leq \mathbb{R}^{d+1}$ such that $\codim(H^\prime)=1$, then $H \subset \mathbb{P}^{d}$ is called a \textit{projective hyperplane}. Let $\mathcal{A}:=\lbrace H_1, ..., H_n \rbrace$ be a finite set of projective hyperplanes. If $\bigcap_{i=1}^n H_i=\lbrace \rbrace$, then $\mathcal{A}$ is called an \textit{arrangement of hyperplanes}, or simply \textit{arrangement} for short.\\
ii) Each arrangement $\mathcal{A}$ induces a cell decomposition $\gamma$ of $\mathbb{P}^{d}$ and two arrangements $\mathcal{A}_1, \mathcal{A}_2$ with corresponding cell decompositions $\gamma_1, \gamma_2$ are called \textit{isomorphic}, if $\gamma_1$ and $\gamma_2$  are isomorphic. Denote by $f^\mathcal{A}_i$ the number of $i$-dimensional cells of $\gamma$ and define $f^\mathcal{A}:=\left(f^\mathcal{A}_0, f^\mathcal{A}_1, ..., f^\mathcal{A}_{d}\right)$, the so called $f$-\textit{vector of} $\mathcal{A}$. The zero-dimensional elements of $\gamma$ are called \textit{vertices} (of $\mathcal{A}$), the one-dimensional elements are called \textit{edges} (of $\mathcal{A}$) and the $d$-dimensional elements are called \textit{chambers} (of $\mathcal{A}$). The set of chambers of $\mathcal{A}$ is denoted by $\mathcal{K}(\mathcal{A})$. Each chamber $C \in \mathcal{K}(\mathcal{A})$ is bounded by at least $d+1$ hyperplanes of $\mathcal{A}$, called the \textit{walls} of $C$. An arrangement is called \textit{simplicial}, if every chamber has precisely $d+1$ walls.\\
iii) To each chamber $C$ we associate a graph $\Gamma^C$ in the following way: the vertices of $\Gamma^C$ are given by the walls of $C$. Two vertices $H, H^\prime$ of $\Gamma^C$ are connected by an edge with weight $|\lbrace k \mid (H \cap H^\prime) \subset H_k \rbrace|$ if said quantity is at least three. The graph $\Gamma^C$ is called the \textit{Coxeter diagram at} $C$. The arrangement $\mathcal{A}$ is called \textit{simply laced} if for any $C \in \mathcal{K}(\mathcal{A})$, the graph $\Gamma^C$ has only edges of weight three.  \\
iv) There is a poset $L:=L(\mathcal{A})$ associated to $\mathcal{A}$, which is defined as follows: let $\mathcal{A}^\prime$ be the set of linear hyperplanes in $\mathbb{R}^{d+1}$ corresponding to $\mathcal{A}$ and let $L$ consist of all subspaces of $\mathbb{R}^{d+1}$ which are obtained as intersections of elements of $\mathcal{A}^\prime$. The \textit{rank} of an element of $L$ is given by its codimension in $\mathbb{R}^{d+1}$.  Ordering $L$ by reverse inclusion makes it into a lattice, the so called \textit{intersection lattice of} $\mathcal{A}$. Denote the M\"obius function of $L$ by $\mu$. Then the \textit{characteristic polynomial} of $\mathcal{A}$ is defined by $\chi(\mathcal{A},t):=\sum_{X \in L} \mu(\mathbb{R}^{d+1},X) t^{\dim(X)}$.\\
v) Elements of $L$ having dimension two are called \textit{lines} of $\mathcal{A}$. Elements of $L$ having dimension one are the same as zero-dimensional elements of $\gamma$ and are also called \textit{vertices} accordingly. Each vertex $v$ is contained in at most $n-1$ hyperplanes of $\mathcal{A}$ and we define $w(v):=|\lbrace i \mid v \in H_i \rbrace|$. The number $w(v)$ is called the \textit{weight} of $v$. The subarrangement of $\mathcal{A}$ consisting of all planes containing $v$ induces a hyperplane arrangement in $\mathbb{P}^{d-1}$ of size $w(v)$, called the \textit{parabolic subarrangement at} $v$.
Similarly, if $l$ is a line, then we define $w(l):=|\lbrace i \mid l \subset H_i \rbrace|$ and call it the \textit{weight} of $l$. For each $i \geq 2$ we denote by $h^\mathcal{A}_i$ the number of lines having weight $i$. Similarly, for each $j \geq 3$ we denote by $t^\mathcal{A}_j$ the number of vertices having weight $j$.
The vector $h^\mathcal{A}:=(h^\mathcal{A}_2, h^\mathcal{A}_3, ...)$ is called the $h$-vector of $\mathcal{A}$ while the vector $t^\mathcal{A}:=(t^\mathcal{A}_3, t^\mathcal{A}_4, ...)$ is called the $t$-vector of $\mathcal{A}$. For fixed $H \in \mathcal{A}$, the set of elements in $L$ which have codimension two and which are contained in $H$ define a hyperplane arrangement in $\mathbb{P}^{d-1}$, called the \textit{restriction of $\mathcal{A}$ to $H$} and denoted by $\mathcal{A}^H$. Finally, we write $m(\mathcal{A})$ for the maximal $j$ such that $t^\mathcal{A}_j>0$ and call it the \textit{multiplicity} of $\mathcal{A}$.
\end{definition}

From now on, we fix $d:=3$, i.e. we study projective hyperplane arrangements in $\mathbb{P}^3$. In particular, the arrangement $\mathcal{A}$ is simply laced precisely when $h^\mathcal{A}_i=0$ for $i\geq4$. If $\mathcal{A}$ is simplicial, then every chamber of $\mathcal{A}$ has exactly four walls, six edges and four vertices. Moreover, for every vertex $v$, the corresponding parabolic subarrangement at $v$ is a simplicial arrangement in $\mathbb{P}^2$ (see for instance Lemma 2.17 in \cite{cuntz_muecksch}). By the same lemma, for every $H \in \mathcal{A}$, the restricted arrangement $\mathcal{A}^H$ is a simplicial arrangement in $\mathbb{P}^2$. In the paper \cite{Gruenbaum_giessen}, it is conjectured that there are only finitely many different isomorphism classes of \textit{irreducible} simplicial arrangements in $\mathbb{P}^3$. Here, the term irreducible means the following: an arrangement $\mathcal{A}$ in $\mathbb{P}^{d}$ is called \textit{reducible}, if there exist $d_1, d_2 \in \mathbb{Z}_{\geq 0}$ and arrangements $\mathcal{A}_1, \mathcal{A}_2$ in $\mathbb{P}^{d_1},\mathbb{P}^{d_2}$ such that $d_1+d_2+1=d$ and $\mathcal{A}=\mathcal{A}_1 \times \mathcal{A}_2$ is a product arrangement (see Chapter 2 in \cite{hyperhyper} for more on this construction). Now, $\mathcal{A}$ is called irreducible if it is not reducible. Observe that $\mathcal{A}$ is irreducible if and only if $\Gamma^C$ is connected for every $C \in \mathcal{K}(\mathcal{A})$ (see Lemma 3.5 in \cite{cuntz_muecksch}). We formulate the following conjecture:

\begin{theoremb} \label{endlich conje}
There are only finitely many different isomorphism classes of irreducible simplicial arrangements in $\mathbb{P}^3$.
\end{theoremb}

We supply some more evidence for Conjecture \ref{endlich conje} by proving that it holds true at least when restricted to simply laced arrangements with splitting characteristic polynomial (see Theorem \ref{simply laced thm}, part ii)). This is achieved by deriving combinatorial restrictions on $h^\mathcal{A},t^\mathcal{A}, f^\mathcal{A}$ and $n$. The obtained inequalities appear to be interesting by themselves. Moreover, most obtained estimates are sharp.\\

In the paper \cite{Gruenbaum_giessen}, we also find the following conjecture involving the $h$-vector of a simplicial arrangement in $\mathbb{P}^3$. 

\begin{theoremb} \label{h_2 conje}
If $\mathcal{A}$ is a simplicial arrangement in $\mathbb{P}^3$, then $h^\mathcal{A}_2 > \sum_{i \geq 3} h^\mathcal{A}_i$.
\end{theoremb}

In Section 4, we verify this conjecture for all currently known irreducible simplicial arrangements in $\mathbb{P}^3$. See also Theorem \ref{simply laced thm}, part iii) and Corollary \ref{h_2 conj crystallo} for related results. In order to simplify the language for the rest of this paper, an arrangement (simplicial or not) which satisfies the conclusion of Conjecture \ref{h_2 conje} will be called a \textit{Gr\"unbaum-Shephard arrangement}.

\end{section}

\begin{section}{Results and proofs}
If not stated otherwise, then throughout this section, $\mathcal{A}$ will always denote a hyperplane arrangement in $\mathbb{P}^3$ consisting of $n$ hyperplanes. If $\chi(\mathcal{A},t)$ splits over $\mathbb{R}$, then $\mathcal{A}$ is sometimes called an  \textit{arrangement with splitting (characteristic) polynomial}. We formulate and prove our two main results (see Theorem \ref{real roots rank 4}, Theorem \ref{simply laced thm}). The first one gives a precise combinatorial condition for the splitting of $\chi(\mathcal{A},t)$ over $\mathbb{R}$. The second theorem deals with simply laced arrangements. It implies that there are only finitely many different isomorphism types of \textit{simplicial} simply laced arrangements with splitting characteristic polynomial. Our techniques will rely mainly on the $h$-vector, $t$-vector and $f$-vector of $\mathcal{A}$. Therefore, most ideas are purely combinatorial in nature. We start right away with the first theorem. 
\begin{theorem} \label{real roots rank 4}
Let $\mathcal{A}$ be a hyperplane arrangement.  Write $h:=\sum_{i \geq 2} (i-1) h^\mathcal{A}_i$ and $g^\mathcal{A}_1:=\sum_{i \geq 2} h^\mathcal{A}_i$. Then the following statements hold:\\
i) The characteristic polynomial of $\mathcal{A}$ is given by $$\chi(\mathcal{A},t)=\left(t-1\right)\left( t^3+(1-n)t^2+(h+1-n)t+ h + 1 -f^\mathcal{A}_3 \right).$$
ii) If all roots of $\chi(\mathcal{A},t)$ are real, then the following relations hold:
\begin{align}
h&=\sum_{i \geq 2} (i-1) h^\mathcal{A}_i= \sum_{H \in \mathcal{A}} |\mathcal{A}^H| -g^\mathcal{A}_1\leq \left \lfloor \frac{(n+2)(n-1)}{3} \right \rfloor , \\
 f^\mathcal{A}_3 &\leq \left \lfloor\frac{(9n+18)h + 20 + 12n + 2 \sqrt{(n^2+n-2-3h )^3}-2n^3-3n^2 }{27}  \right \rfloor, \\
f^\mathcal{A}_3 &\geq \left \lceil \frac{(9n+18)h + 20 + 12n - 2 \sqrt{(n^2+n-2-3h )^3}-2n^3-3n^2 }{27}  \right \rceil.  
\end{align}
Moreover, the estimates given in (1),(2),(3) are tight.
\end{theorem}

\begin{proof}
i) Let $L:=L(\mathcal{\mathcal{A}})$ be the intersection lattice of $\mathcal{A}$. We denote the M\"obius Function of $L$ by $\mu$ and we set $\chi:=\chi(\mathcal{A},t)$. The formula may be deduced from Zaslavsky's Theorem (see Theorem 2.68 in \cite{hyperhyper}) and the fact that $\chi(1)=0$. By said theorem, we know that $\chi(-1)=2f^\mathcal{A}_3$.  Observe that the values $\mu(\mathbb{R}^4,X)$ are known for any $X \in L$ having codimension at most two (see for instance Chapter 2 in \cite{hyperhyper}), giving $\chi=\sum_{X \in L} \mu(\mathbb{R}^4,X) t^{\dim(X)} =t^4-nt^3+ht^2+at+b$ for certain $a, b \in \mathbb{Z}$. The values $a,b$ are obtained by solving the system of linear equations imposed by the conditions $\chi(1)=0, \chi(-1)=2f^\mathcal{A}_3$. Doing so, we obtain $\chi=t^4-n t^3+h t^2 + (n-f^\mathcal{A}_3) t +f^\mathcal{A}_3-1-h$. Using polynomial division to divide $\chi$ by $t-1$, we obtain the claimed formula.\\
ii) The first two equalities in (1) hold by definition of $h$ and $g^\mathcal{A}_1$. By part i), we see that $\chi(\mathcal{A},t)$ splits over $\mathbb{R}$ if and only if $\tilde{\chi}:=\frac{\chi(\mathcal{A},t)}{t-1}$ splits over $\mathbb{R}$. As $\tilde{\chi}$ is a cubic polynomial, this can happen only if the discriminant of $\tilde{\chi}$ is non-negative. Explicit computation using the formula in i) gives the conditions (1),(2) and (3).
We note that all given estimates are tight for the arrangement $A^3_2(15)$ (see \cite{Gruenbaum_giessen} and Table 1). 
\end{proof}

Before moving on towards Theorem \ref{simply laced thm}, we draw some corollaries from Theorem \ref{real roots rank 4}.

\begin{corollary} \label{chamberbound h relatio}
Assume that all the roots of $\chi(\mathcal{A},t)$ are real. \\
i) The number of chambers is bounded by  $f^\mathcal{A}_3 \leq \frac{(n+2)^3}{27}$. \\
ii) We have the following inequality: $$\sum_{i \geq 3} (i-1)(i-2) h^\mathcal{A}_i \geq \frac{(n-4)(n-1)}{3}.$$ In particular, for $n \geq 5$ there exists $j \geq 3$ such that $h^\mathcal{A}_j>0$, i.e. there exists a line which is contained in at least three planes of $\mathcal{A}$.  
\end{corollary}

\begin{proof}
i) We observe that for $n \geq 4$ the function $\phi_n:\mathbb{R}_{>0} \longrightarrow \mathbb{R}$, defined by $\phi_n(x):= (9n+18)x + 20 + 12n + 2 \sqrt{(n^2+n-2-3x )^3}-2n^3-3n^2$, takes its maximal value at $\overline{x}=\frac{n^2+n-2}{3}$. One computes $\phi_n(\overline{x})=(n+2)^3$. Thus, the claim follows using relation (2) from Theorem \ref{real roots rank 4}.\\
ii) We note that $\sum_{i \geq 2} \binom {i}{2} h^\mathcal{A}_i = \binom{n}{2}$. Using this, relation (1) gives $$\sum_{i \geq 3} (i-1)(i-2) h^\mathcal{A}_i=2 \sum_{i \geq 2} \left(\binom{i}{2}-(i-1)\right)h^\mathcal{A}_i \geq\frac{(n-4)(n-1)}{3}.$$ The last statement is obvious. 
\end{proof}

\begin{remark}
Let $\mathcal{A}^\prime$ be an arrangement in $\mathbb{P}^2$ such that $\chi(\mathcal{A}^\prime,t)$ splits over $\mathbb{R}$. It is easy to see that in this case (and only in this case) one has $f^{\mathcal{A}^\prime}_2 \leq \frac{(n+1)^2}{4}$. Now, let $d \geq 3$ be a natural number and assume that $\mathcal{A}$ is an arrangement in $\mathbb{P}^{d-1}$ whose characteristic polynomial has only real roots. Inspired by Corollary \ref{chamberbound h relatio}, part i) and the above observation, we conjecture that \begin{align}
f^\mathcal{A}_{d-1} \leq \left(1 + \frac{n-1}{d-1}\right)^{d-1}.
\end{align} We note that the conjecture holds for all arrangements presented in \cite{Cuntz4space} (it is not hard to see that the characteristic polynomials of the given arrangements split over $\mathbb{R}$). Moreover, all arrangements in \cite{Cuntz4space} are simplicial and as simplicial arrangements tend to have many chambers (one has $ f^\mathcal{A}_{d-1} \leq \frac{2 f^\mathcal{A}_{d-2}}{d} $ with equality in the simplicial case), this is quite good evidence for the truth of (4) for an arbitrary arrangement with splitting polynomial.      
\end{remark}

The following corollary allows us to rephrase Theorem 1 in the simplicial case, using only the numbers $h^\mathcal{A}_i, t^\mathcal{A}_j$ and $n$.

\begin{corollary}  \label{simp t vek coroll}
Assume that $\mathcal{A}$ is simplicial. We write $\tilde{\chi}(\mathcal{A},t):= \frac{\chi(\mathcal{A},t)}{t-1}$, $m:=m(\mathcal{A})$ and $h:=\sum_{i=2}^{m-1} (i-1)h^\mathcal{A}_i$. Then $f^\mathcal{A}_3+n=\sum_{i = 3}^m i t^\mathcal{A}_i$ and  $$\tilde{\chi}(\mathcal{A},t)=
t^3 + (1-n) t^2 + \left( 1+\sum_{i = 2}^{m-1} (i-1) h^\mathcal{A}_i \right) \left(t+1\right) +n(1-t) - \sum_{i = 3}^m i t^\mathcal{A}_i.$$ Moreover, if all roots of $\chi(\mathcal{A},t)$ are real, then $\sum_{i = 3}^m i t^\mathcal{A}_i \leq \frac{(n+2)^3}{27} + n$ and relations (2),(3) from Theorem \ref{real roots rank 4} take the following forms: \begin{align*}
\sum_{i \geq 3} i t^\mathcal{A}_i &\leq \left \lfloor\frac{(9n+18)h + 20 + 39n + 2 \sqrt{(n^2+n-2-3h )^3}-2n^3-3n^2 }{27}  \right \rfloor, \\
\sum_{i \geq 3} i t^\mathcal{A}_i &\geq \left \lceil \frac{(9n+18)h + 20 + 39n - 2 \sqrt{(n^2+n-2-3h )^3}-2n^3-3n^2 }{27}  \right \rceil.
\end{align*}  
\end{corollary}

\begin{proof}
As $\mathcal{A}$ is simplicial, we have $f^\mathcal{A}_3=\frac{1}{2} f^\mathcal{A}_2= \frac{1}{2} \sum_{H \in \mathcal{A}} |\mathcal{K}(\mathcal{A}^H)|$. Let $t^H$ denote the $t$-vector of the arrangement $\mathcal{A}^H$. Then one has:\\ $\sum_{H \in \mathcal{A}}  |\mathcal{K}(\mathcal{A}^H)|=\sum_{H \in \mathcal{A}} \left(-2 +\sum_{i \geq 2}  2t^H_i \right)=2\left( -n + \sum_{H \in \mathcal{A}, i \geq 2} t^H_i \right)$. \\
By double counting, we conclude that $f^\mathcal{A}_3=-n + \sum_{i \geq 3} it^\mathcal{A}_i$. The remaining claims all follow from Theorem \ref{real roots rank 4} and Corollary \ref{chamberbound h relatio}.
\end{proof}

We now move towards our result on simply laced arrangements. In order to prove it, we need one more lemma.   

\begin{lemma} \label{parabol simply laced}
Let $\mathcal{A}$ be a simplicial hyperplane arrangement. If $\mathcal{A}$ is simply laced, then the same is true for every parabolic subarrangement of $\mathcal{A}$. In particular, we have $m(\mathcal{A}) \leq 7$. 
\end{lemma}

\begin{proof}
Let $\mathcal{B}$ be a parabolic subarrangement of $\mathcal{A}$. Then for any chamber $C \in \mathcal{K}(\mathcal{B})$, the corresponding Coxeter diagram $\Gamma^C$ is obtained as a subgraph of the Coxeter diagram $\Gamma^{C^\prime}$, where $C^\prime \in \mathcal{K}(\mathcal{A})$ is a suitable chamber of $\mathcal{A}$ (see for instance Lemma 3.7 in \cite{cuntz_muecksch}). In particular, this proves that $\mathcal{B}$ is simply laced, i.e. every vertex of $\mathcal{B}$ is contained in at most three lines of $\mathcal{B}$. Now fix a chamber  $C \in \mathcal{K}(\mathcal{B})$. First assume that $\overline{C}$ contains two vertices which are contained in precisely two lines of $\mathcal{B}$. Then Lemma 2 from paper \cite{free_simp_rang3} shows that $\mathcal{B}$ is a near pencil arrangement. In particular, we have $|\mathcal{B}| \leq 4$. Now assume that every chamber of $\mathcal{B}$ contains at most one vertex which is contained in precisely two lines of $\mathcal{B}$. As all vertices of $\mathcal{B}$ are contained in at most three lines, we may apply Lemma 3 in \cite{free_simp_rang3} to obtain $|\mathcal{B}| \leq 7$. 
\end{proof}

\begin{remark}
If $\mathcal{A}$ is an irreducible simplicial arrangement, then one has $m(\mathcal{A}) \geq 6$. To see this, we argue as follows: as $\mathcal{A}$ is assumed to be irreducible, we know by Lemma 3.5 in \cite{cuntz_muecksch} that the Coxeter diagram $\Gamma^C$ is connected for any $C \in \mathcal{K}(\mathcal{A})$; in particular, $\Gamma^C$ contains at least two connected subgraphs on three vertices for any $C \in \mathcal{K}(\mathcal{A})$. Now denote by $i(C)$ the number of vertices $v \in C$ such that $\mathcal{A}_v$ is irreducible. We apply Lemma 3.5 from \cite{cuntz_muecksch} once more and obtain $0 < 2 f^\mathcal{A}_3 \leq \sum_{C \in \mathcal{K}(\mathcal{A})} i(C) \leq \sum_{j \geq 6} \lambda_j t^\mathcal{A}_j$ for certain non-negative $\lambda_j \in \mathbb{Z}$. Here, the last inequality holds because every parabolic subarrangement $\mathcal{A}_v$ of size at most five is reducible (see for instance \cite{Gruenbaum}).  
\end{remark}

\begin{theorem} \label{simply laced thm}
Assume that $\mathcal{A}$ is simply laced and that all the roots of $\chi(\mathcal{A},t)$ are real. Then the following is true: \\
i) We have $h^\mathcal{A}_2 \leq 2n-2$ and $h^\mathcal{A}_3 \geq \frac{(n-4)(n-1)}{6}$.  
Moreover, one has \begin{align}
\frac{(n+2)^3}{27} \geq f^\mathcal{A}_3 \geq \frac{n^3+6n+20+3 h^\mathcal{A}_2 (n+2)-2 \sqrt{\left(2n-2-h^\mathcal{A}_2\right)^3} }{27}.
\end{align} In particular, we have $\lim_{n \to \infty} \frac{f^\mathcal{A}_3}{n^3}=\frac{1}{27}$ and (5) may be considered asymptotically optimal.\\
ii) If $\mathcal{A}$ is simplicial, then $n \leq 119$. \\
iii) If $\mathcal{A}$ is a Grünbaum-Shephard arrangement, then $n \leq 15$.
\end{theorem}

\begin{proof}
i) The lower bound for $h^\mathcal{A}_3$ follows from Corollary \ref{chamberbound h relatio}, part ii). Using relation (1), we obtain $\frac{(n+2)(n-1)}{3} \geq h^\mathcal{A}_2+2 h^\mathcal{A}_3 \geq h^\mathcal{A}_2 + \frac{(n-4)(n-1)}{3}$, proving the upper bound for $h^\mathcal{A}_2$. In order to obtain (5), remember that $h^\mathcal{A}_2 + 3 h^\mathcal{A}_3 = \binom{n}{2}$. Solving for $h^\mathcal{A}_2$, we obtain $h=h^\mathcal{A}_2 + 2 h^\mathcal{A}_3= \frac{n^2-n+h^\mathcal{A}_2}{3}$. Inequality (5) is now obtained by substituting the obtained expression for $h$ in inequality (3).  \\
ii) By Lemma \ref{parabol simply laced} we know that $m(\mathcal{A}) \leq 7$. Now, for each chamber $C \in \mathcal{K}(\mathcal{A})$, we denote by $C_i$ the number of edges of $\overline{C}$ which are contained in a line of weight $i$. We obtain the following estimate: \begin{align}
f^\mathcal{A}_3 =  \sum_{C \in \mathcal{K}(\mathcal{A})} \frac{C_2 + C_3}{6} \geq \sum_{C \in \mathcal{K}(\mathcal{A})} \frac{C_3}{6}	 \geq \frac{(n-3)h^\mathcal{A}_3}{4}  \geq \frac{ (n-4)(n-3)(n-1)}{24}.
\end{align} For this observe that every chamber has precisely six edges. Furthermore, every line of weight three contains at least $\frac{n-3}{4}$ edges (every vertex is incident with at most seven planes). Moreover, any such edge is contained in six chambers and we have $h^\mathcal{A}_3 \geq \frac{(n-4)(n-1)}{6}$ by Corollary \ref{chamberbound h relatio}, part ii). By part i) of the same corollary, we know that there are at most $\frac{(n+2)^3}{27}$ chambers, i.e. we have  $\frac{(n+2)^3}{27} \geq f^\mathcal{A}_3 \geq \frac{(n-4)(n-3)(n-1)}{24}$. This is possible only for $n \leq 119$.\\
iii) If $\mathcal{A}$ is a Gr\"unbaum-Shephard arrangement, then by part i) we have $2n-2 \geq h^\mathcal{A}_2 > h^\mathcal{A}_3 \geq \frac{(n-4)(n-1)}{6}$, forcing $n \leq 15$.
\end{proof}

\begin{remark}
The simplicial arrangement $A^3_2(15)$ shows that the upper bound for $n$ given in Theorem \ref{simply laced thm}, part iii) is sharp (see Table 1). 
\end{remark}

We observe that inequality (6) immediately generalizes to the following result, which appears to be interesting in its own right.

\begin{corollary} \label{f_3 h abschaetzung}
Assume that $\mathcal{A}$ is simplicial and write $m:=m( \mathcal{A})$. Then we have $ \sum_{i=3}^m i t^\mathcal{A}_i  \geq n + \sum_{i = 2}^{m-1} \frac{i(n-i)}{3(m-i)} h^\mathcal{A}_i 
$.  
\end{corollary}

\begin{proof}
We use the same argument as in the proof of Theorem \ref{simp t vek coroll}, part ii). As $\mathcal{A}$ is simplicial, every chamber of $\mathcal{A}$ has six edges and every edge is contained in a line of weight $i$ for some $i \geq 2$. Moreover, every line of weight $i$ contains at least $\frac{n-i}{m-i}$ edges. Finally, every edge which is contained in a line of weight $i$ is contained in precisely $2i$ chambers of $\mathcal{A}$. This implies $$6 f^\mathcal{A}_3 \geq \sum_{i = 2}^{m-1} \frac{2i(n-i)}{m-i} h^\mathcal{A}_i 
$$ and the claim follows using Corollary \ref{simp t vek coroll}.
\end{proof}

We close this section with the following result, which classifies irreducible simplicial arrangements in $\mathbb{P}^3$ having the smallest possible multiplicity.

\begin{theorem}
Assume that $\mathcal{A}$ is an irreducible simplicial arrangement. If $m(\mathcal{A})=6$, then $\mathcal{A}$ is of type $\mathcal{A}(A_4)$ or $\mathcal{A}(D_4)$. In particular, $\mathcal{A}$ is a simply laced Gr\"unbaum-Shephard arrangement with splitting polynomial.
\end{theorem}

\begin{proof}
a) Every irreducible parabolic subarrangement of $\mathcal{A}$ is of type $A(6,1)$:
if $\mathcal{B}$ is such an arrangement, then by assumption we have $|\mathcal{B}|=6$. Using the enumeration provided in \cite{Cuntz27}, we conclude that $\mathcal{B}$ must be of type $A(6,1)$. \\
b) There is a fixed Coxeter diagram $\Gamma$ such that $\Gamma=\Gamma^C$ for every $C \in \mathcal{K}(\mathcal{A})$: let $C \in \mathcal{K}(\mathcal{A})$ be a chamber. Then by part a), the graph $\Gamma^C$ is either of type $A_4$, $D_4$ or it is of type $\tilde{A_3}$. Using this, the claim follows from Lemma 2.31 and Lemma 3.4 in \cite{cuntz_muecksch}. \\
c) Only Coxeter diagrams of type $A_4$, $D_4$ can appear as $\Gamma^C$ for $C \in \mathcal{K}(\mathcal{A})$: note that by Theorem 2.5 in \cite{cuntz_muecksch}, every simplicial arrangement in $\mathbb{P}^3$ has a reducible parabolic subarrangement.  Using part b), we conclude that for any $C \in \mathcal{K}(\mathcal{A})$, the Coxeter diagram $\Gamma^C$ must be of type $A_4$ or $D_4$.  \\
d) If $\Gamma^C$ is of type $D_4$ for every $C \in \mathcal{K}(\mathcal{A})$, then $\mathcal{A}$ is of type $\mathcal{A}(D_4)$: every chamber $C$ contains exactly one vertex of weight three and three vertices of weight six. In particular, we have $t^\mathcal{A}_4=t^\mathcal{A}_5=0$, $t^\mathcal{A}_3= t^\mathcal{A}_6$, and $f^\mathcal{A}_3=8 t^\mathcal{A}_3$. By Corollary \ref{simp t vek coroll}, we have $3 t^\mathcal{A}_4 + 6 t^\mathcal{A}_6 = f^\mathcal{A}_3 + n$. We obtain $t^\mathcal{A}_3=t^\mathcal{A}_6=n$ and $f^\mathcal{A}_3= 8n$. Furthermore, every chamber has exactly three edges of weight two and three edges of weight three. Thus, using similar techniques as in the proof of Corollary \ref{f_3 h abschaetzung}, we obtain: $24n=3 f^\mathcal{A}_3 \geq (n-2)h^\mathcal{A}_2 $, $24n=3 f^\mathcal{A}_3 \geq 2 (n-3)h^\mathcal{A}_3 $. We conclude that $\frac{24n}{n-2}+\frac{36n}{n-3} \geq h^\mathcal{A}_2 + 3 h^\mathcal{A}_3=\frac{n^2-n}{2}$, proving that $n \leq 12$. Next, we prove that $n \in \lbrace 10, 11 \rbrace$ is impossible. Assume that $n=10$ and pick a vertex $v$ of weight three. Then there exists a line $\ell$ of weight two such that $v \in \ell$. By the pigeon pole principle, we conclude that $\ell$ contains at least one edge which is bounded by two vertices of weight three. In particular, there exists $H \in \mathcal{A}$ such that $\ell \in \mathcal{A}^H$. But then $\mathcal{A}^H$ is reducible, contradicting the irreducibility of $\mathcal{A}$. Now assume that $n=11$ and choose a line $\ell$ of weight three. Then $\ell$ contains only vertices of weight six. Clearly, $\ell$ must contain at least $\left \lceil \frac{11-3}{6-3} \right \rceil=\left \lceil\frac{8}{3}\right \rceil=3$ such vertices. However, this is possible only for $n \geq 12$, another contradiction. It is not hard to check that $n \leq 9$ is impossible as well. We conclude that $n=12$. Next, observe that in this case every vertex of weight three is contained in precisely three lines of weight two and every line of weight two contains exactly two vertices of weight three and two vertices of weight six; similarly, every line of weight three contains precisely three vertices of weight six. Using Lemma 3 from \cite{free_simp_rang3}, this implies that for every $H \in \mathcal{A}$, the restricted arrangement $\mathcal{A}^H$ is either of type $A(6,1)$ or $A(7,1)$. However, as $96=f^\mathcal{A}_3=\frac{1}{2} \sum_{H \in \mathcal{A}} |\mathcal{K}(\mathcal{A}^H)|$, we conclude that every restriction must be of type $A(7,1)$. This implies that $\mathcal{A}$ must be of type $\mathcal{A}(D_4)$. \\
e) If $\Gamma^C$ is of type $A_4$ for every $C \in \mathcal{K}(\mathcal{A})$, then $\mathcal{A}$ itself is of type $\mathcal{A}(A_4)$: every chamber $C$ contains exactly two vertices of weight four and two vertices of weight six. In particular, we have $t^\mathcal{A}_3=t^\mathcal{A}_5=0$, $t^\mathcal{A}_4=2 t^\mathcal{A}_6$, and $f^\mathcal{A}_3=6 t^\mathcal{A}_4$. By Corollary \ref{simp t vek coroll}, we have $4 t^\mathcal{A}_4 + 6 t^\mathcal{A}_6 = f^\mathcal{A}_3 + n$. We obtain $t^\mathcal{A}_4=n, t^\mathcal{A}_6=\frac{n}{2}, f^\mathcal{A}_3= 6n$. Furthermore, every chamber has exactly three edges of weight two and three edges of weight three. Thus, by applying similar techniques as in the proof of Corollary \ref{f_3 h abschaetzung}, we obtain: $18n=3 f^\mathcal{A}_3 \geq (n-2)h^\mathcal{A}_2 $, $18n=3 f^\mathcal{A}_3 \geq 2 (n-3)h^\mathcal{A}_3 $. However, this implies $\frac{18n}{n-2}+\frac{27n}{n-3} \geq h^\mathcal{A}_2 + 3 h^\mathcal{A}_3=\frac{n^2-n}{2}$, showing that $n \leq 11$. As $\frac{n}{2}=t^\mathcal{A}_6 \in \mathbb{Z}$, we see that $n=11$ is impossible. It is easy to check that $n \leq 9$ is impossible as well. We conclude that $n=10$. In particular, we have $60=f^\mathcal{A}_3=\frac{1}{2} \sum_{H \in \mathcal{A}} |\mathcal{K}(\mathcal{A}^H)| \geq \frac{10 \cdot 12}{2}=60$. This is true because for every $H \in \mathcal{A}$, the restriction $\mathcal{A}^H$ is an irreducible arrangement in $\mathbb{P}^2$. The unique(!) irreducible simplicial arrangement in $\mathbb{P}^2$ with the smallest number of chambers is the arrangement $A(6,1)$, which has exactly $12$ chambers. We conclude that for every $H \in \mathcal{A}$, the restricted arrangement $\mathcal{A}^H$ is of type $A(6,1)$, showing that $\mathcal{A}$ must be of type $\mathcal{A}(A_4)$.     
\end{proof}
\end{section}

\begin{section}{An updated catalogue of simplicial arrangements in $\mathbb{P}^3(\mathbb{R})$}
As we mentioned in the introduction, there are catalogues published by Gr\"unbaum and Shephard, listing all isomorphism classes of (irreducible) simplicial arrangements in $\mathbb{P}^2$ and $\mathbb{P}^3$ that were known at the time (see \cite{Gruenbaum}, \cite{Gruenbaum_giessen}). In the paper \cite{Cuntz27}, a corrected version for the catalogue presented in \cite{Gruenbaum} was given. In this section, we will provide a corrected version for the catalogue presented in \cite{Gruenbaum_giessen}. However, in sharp contrast to \cite{Cuntz27}, we \textit{do not claim} that our catalogue contains all arrangements up to a certain size. It only reflects the author's current state of knowledge. All examples that are missing in \cite{Gruenbaum_giessen}, can be found in the paper \cite{Cuntz4space}. So our catalogue at least contains all \textit{crystallographic} examples (see \cite{Cuntz4space} for a definition). First, we recall that if $\mathcal{A}, \mathcal{A}^\prime$ are simplicial arrangements in $\mathbb{P}^{d}, \mathbb{P}^{d^\prime}$ respectively, then the product arrangement $\mathcal{A} \times \mathcal{A}^\prime$ is a simplicial arrangement in $\mathbb{P}^{d+d^\prime+1}$. Therefore, we will only list irreducible examples, i.e. examples of arrangements that cannot be obtained by the product construction described above. For each arrangement, we will list its $h$-vector, its $t$-vector and its $f$-vector. Moreover, some comments are provided, mainly regarding differing terminologies in \cite{Gruenbaum_giessen} and \cite{Cuntz4space}. We decide to adopt the notation from \cite{Gruenbaum_giessen}. The data shows that all arrangements are Gr\"unbaum-Shephard arrangements, in accordance with Conjecture \ref{h_2 conje} (see also Corollary \ref{h_2 conj crystallo}). Moreover, all appearing characteristic polynomials split over $\mathbb{R}$, as can be checked from the data via Theorem \ref{real roots rank 4} (or Corollary \ref{simp t vek coroll}). 
Finally, we remark that normal vectors for all crystallographic arrangements can be found in \cite{Cuntz4space} while the root vectors of reflection arrangements should be well known. For the two remaining arrangements of type $A^3_1(27),A^3_1(28)$, normal vectors are provided in the appendix. 

\begin{center}
    \begin{tabular}{|p{0.8cm} | p{2cm} |p{3.3cm} | p{2.8cm} | p{1.9cm} |}
    \hline
    \normalfont{label} & \normalfont{$h$-vector} & \normalfont{$t$-vector} & \normalfont{$f$-vector}& \normalfont{comments}  \\ \hline
 \begin{tiny} $A^3_1(10)$ \end{tiny}    &\begin{tiny}
 (15, 10)
\end{tiny}   &\begin{tiny}
(0, 10, 0, 5)
\end{tiny}  & \begin{tiny}
(15, 75, 120, 60)
\end{tiny} & \begin{tiny}
type $\mathcal{A}(A_4)$
\end{tiny} \\ \hline
  \begin{tiny}  $A^3_1(12)$ \end{tiny}&\begin{tiny}
  (18, 16)
\end{tiny}      & \begin{tiny}
(12, 0, 0, 12)
\end{tiny}  & \begin{tiny}
(24, 120, 192, 96)
\end{tiny} & \begin{tiny}
type $\mathcal{A}(D_4)$
\end{tiny} \\ \hline
  \begin{tiny}  $A^3_1(13)$ \end{tiny}& \begin{tiny}
  (21, 19)
\end{tiny}    &\begin{tiny}
(6, 10, 0, 9, 3)
\end{tiny}   &\begin{tiny}
(28, 148, 240, 120)
\end{tiny} &\begin{tiny}
subarrangement of $\mathcal{A}( B_4 )$
\end{tiny}\\ \hline 
 \begin{tiny}   $A^3_1(14)$ \end{tiny}&\begin{tiny}
 (25, 20, 1)
\end{tiny}   & \begin{tiny}
(2, 16, 2, 8, 2, 2)
\end{tiny} &\begin{tiny}
(32, 176, 288, 144)
\end{tiny} &\begin{tiny}
subarrangement of $\mathcal{A}( B_4 )$
\end{tiny}\\ \hline
  \begin{tiny}  $A^3_1(15)$ \end{tiny}&\begin{tiny}
  (30, 19, 3)
\end{tiny}    & \begin{tiny}
(0, 18, 6, 8, 0, 3, 1)
\end{tiny}  & \begin{tiny}
(36, 204, 336, 168)
\end{tiny} & \begin{tiny}
subarrangement of $\mathcal{A}( B_4 )$
\end{tiny}   \\ \hline    
\begin{tiny}	$A^3_2(15)$ \end{tiny}& \begin{tiny}
(27, 26)\end{tiny} &\begin{tiny}
(0, 24, 0, 6, 9)\end{tiny}   & \begin{tiny}
(39, 219, 360, 180)\end{tiny} & \begin{tiny} Nr. 1 in  \cite{Cuntz4space} \end{tiny}\\ \hline
\begin{tiny}	$A^3_1(16)$ \end{tiny}&\begin{tiny} (36, 16, 6)  \end{tiny}&\begin{tiny} (0, 16, 12, 8, 0, 0, 4) \end{tiny} &\begin{tiny} (40, 232, 384, 192) \end{tiny} &\begin{tiny}
type $\mathcal{A}(B_4)$
\end{tiny} \\ \hline
\begin{tiny}	$A^3_1(17)$ \end{tiny}&\begin{tiny} (34, 28, 3)  \end{tiny}&\begin{tiny} (12, 20, 0, 14, 0, 6, 1) \end{tiny} &\begin{tiny} (53, 293, 480, 240) \end{tiny} &\begin{tiny}
Nr. 2 in \cite{Cuntz4space}
\end{tiny} \\ \hline
\begin{tiny}	$A^3_1(18)$ \end{tiny}&\begin{tiny} (39, 32, 3)  \end{tiny}&\begin{tiny} (0, 36, 3, 8, 6, 6, 1) \end{tiny} &\begin{tiny} (60, 348, 576, 288) \end{tiny} &\begin{tiny}
Nr. 3 in \cite{Cuntz4space}
\end{tiny} \\ \hline
\begin{tiny}	$A^3_1(21)$ \end{tiny}&\begin{tiny} (51, 41, 6)  \end{tiny}&\begin{tiny} (12, 38, 6, 21, 3, 6, 0, 4) \end{tiny} &\begin{tiny} (90, 522, 864, 432) \end{tiny} &\begin{tiny}
missing in \cite{Gruenbaum_giessen}; Nr. 4 in \cite{Cuntz4space}
\end{tiny} \\ \hline
\begin{tiny}	$A^3_1(22)$ \end{tiny}&\begin{tiny} (57, 40, 9)  \end{tiny}&\begin{tiny} (12, 48, 6, 20, 0, 6, 4, 4) \end{tiny} &\begin{tiny} (100, 580, 960, 480) \end{tiny} &\begin{tiny}
missing in \cite{Gruenbaum_giessen}; Nr. 5 in \cite{Cuntz4space}
\end{tiny} \\ \hline
\begin{tiny}	$A^3_1(24)$ \end{tiny}&\begin{tiny} (72, 32, 18)  \end{tiny}&\begin{tiny} (0, 96, 0, 0, 0, 0, 24) \end{tiny} &\begin{tiny} (120, 696, 1152, 576) \end{tiny} &\begin{tiny}
type $\mathcal{A}(F_4)$; \newline Nr. 6 in \cite{Cuntz4space}
\end{tiny} \\ \hline
\begin{tiny}	$A^3_1(25)$ \end{tiny}&\begin{tiny} (75, 55, 10)  \end{tiny}&\begin{tiny} (0, 60, 30, 25, 15, 0, 0, 10) \end{tiny} &\begin{tiny} (140, 860, 1440, 720) \end{tiny} &\begin{tiny}
missing in \cite{Gruenbaum_giessen}; Nr. 7 in \cite{Cuntz4space}
\end{tiny} \\ \hline
    \begin{tiny}	$A^3_1(27)$ \end{tiny}&\begin{tiny} (81, 70, 0, 6)  \end{tiny}&\begin{tiny} (30, 60, 0, 67, 0, 0, 0, 12, 0, 0, 0, 0, 1) \end{tiny} &\begin{tiny} (170, 1010, 1680, 840) \end{tiny} &\begin{tiny}
subarrangement of $\mathcal{A}( H_4 )$
\end{tiny} \\ \hline
\begin{tiny}	$A^3_1(28)$ \end{tiny}&\begin{tiny} (90, 76, 0, 6)  \end{tiny}&\begin{tiny} (0, 100, 0, 58, 15, 0, 0, 12, 0, 0, 0, 0, 1) \end{tiny} &\begin{tiny} (186, 1146, 1920, 960) \end{tiny} &\begin{tiny}
subarrangement of $\mathcal{A}( H_4 )$
\end{tiny} \\ \hline
\end{tabular}
\end{center}
\begin{center}
    \begin{tabular}{|p{0.8cm} | p{2cm} |p{3.3cm} | p{2.8cm} | p{1.9cm} |}
    \hline
    \normalfont{label} & \normalfont{$h$-vector} & \normalfont{$t$-vector} & \normalfont{$f$-vector}& \normalfont{comments}  \\ \hline
\begin{tiny}	$A^3_2(28)$ \end{tiny}&\begin{tiny} (90, 64, 16)  \end{tiny}&\begin{tiny} (24, 84, 18, 40, 0, 18, 3, 0, 6, 0, 1) \end{tiny} &\begin{tiny} (194, 1154, 1920, 960) \end{tiny} &\begin{tiny}
Nr. 8 in \cite{Cuntz4space}
\end{tiny} \\ \hline
\begin{tiny}	$A^3_1(30)$ \end{tiny}&\begin{tiny} (99, 84, 9, 0, 2)  \end{tiny}&\begin{tiny} (0, 144, 0, 36, 24, 18, 0, 0, 0, 0, 6) \end{tiny} &\begin{tiny} (228, 1380, 2304, 1152) \end{tiny} &\begin{tiny}
Nr. 9 in \cite{Cuntz4space}
\end{tiny} \\ \hline
\begin{tiny}	$A^3_1(32)$ \end{tiny}&\begin{tiny} (120, 76, 18, 4)  \end{tiny}&\begin{tiny} (24, 120, 24, 68, 0, 6, 10, 8, 0, 0, 6) \end{tiny} &\begin{tiny} (266, 1610, 2688, 1344) \end{tiny} &\begin{tiny}
missing in \cite{Gruenbaum_giessen}; Nr. 10 in \cite{Cuntz4space}
\end{tiny} \\ \hline
\begin{tiny}	$A^3_2(32)$ \end{tiny}&\begin{tiny} (124, 64, 30)  \end{tiny}&\begin{tiny} (0, 144, 48, 40, 0, 0, 12, 16, 0, 0, 4) \end{tiny} &\begin{tiny} (264, 1608, 2688, 1344) \end{tiny} &\begin{tiny}
missing in \cite{Gruenbaum_giessen}; Nr. 11 in \cite{Cuntz4space}
\end{tiny} \\ \hline
\begin{tiny}	$A^3_1(60)$ \end{tiny}&\begin{tiny} (450, 200, 0, 72)  \end{tiny}&\begin{tiny} (0, 600, 0, 660, 0, 0, 0, 0, 0, 0, 0, 0, 60) \end{tiny} &\begin{tiny} (1320, 8520, 14400, 7200) \end{tiny} &\begin{tiny}
type $\mathcal{A}(H_4)$
\end{tiny} \\ \hline
    \end{tabular}
    \captionof{table}{The list of currently known irreducible simplicial hyperplane arrangements in $\mathbb{P}^3(\mathbb{R})$.}
\end{center}

\begin{remark}
As noted above, the given table shows that all \textit{known} simplicial arrangements in $\mathbb{P}^3$ have a characteristic polynomial which has only real roots.
We observe that for simplicial arrangements in $\mathbb{P}^2$, this is not the case: several counterexamples can be found for instance in the paper \cite{Cuntz27}. In said paper, the smallest counterexample is denoted by $A(13,4)$; this arrangement also arises as extremizer for the so called Dirac-Motzkin Conjecture: the statement that an arrangement of $n$ lines in $\mathbb{P}^2$ determines at least $\left \lfloor \frac{n}{2} \right \rfloor$ vertices of weight two. In the paper \cite{greentao}, this conjecture was proved to be a theorem, at least for sufficiently large arrangements. Moreover, in the paper \cite{free_simp_rang3}, we prove that the conjecture holds for \textit{any} arrangement in $\mathbb{P}^2$ whose characteristic polynomial splits over $\mathbb{R}$. In fact, the statement is not only proved for straight line arrangements, but also for arrangements of \textit{pseudolines} with splitting polynomial.
\end{remark}

The computed data combined with the results obtained in the paper \cite{Cuntz4space} immediately shows that Conjecture \ref{h_2 conje} is a theorem for crystallographic arrangements: 

\begin{corollary} \label{h_2 conj crystallo}
Every crystallographic arrangement in $\mathbb{P}^3$ is a Gr\"unbaum-Shephard arrangement.
\end{corollary}

\end{section}

\begin{section}{Appendix}
The purpose of this short appendix is to provide normal vectors for the arrangements of type $A^3_1(27), A^3_1(28)$ respectively. We write $\tau:=\frac{1+\sqrt{5}}{2}$. Then the arrangement $A^3_1(28)$ can be defined by the following normal vectors: \\\begin{footnotesize}
$\lbrace \left( 1, 0, 0, 0 \right), \left( 0, 1, 0, 0 \right),\left( 0, 1, 1, 0 \right), \left( 0, 0, 1, 0 \right), \left( 0, 1, 1, 1 \right), \left( 0, 0, 1, 1 \right),\\ \left( 0, 0, 0, 1 \right), \left( 1, 1, 0, 0 \right), \left( 1, 1, 1, 1 \right), \left( 1, \tau, 0, 0 \right) 
  , \left( \tau, 1, 0, 0 \right), 
  \left( 1, \tau, \tau, \tau \right),  \\
  \left( \tau, 1, 1, 0 \right),  
  \left( \tau, 1+\tau, 1, 0 \right), 
  \left( \tau, 1, 1, 1 \right), 
  \left( \tau, 1+\tau, \tau, 
      \tau \right), 
  \left( 1+\tau, 1+\tau, 1, 0 \right),  \\
  \left( \tau, 1+\tau, 1, 1 \right), 
  \left( 1+\tau, 1+\tau, 
      \tau, \tau \right), 
  \left( \tau, 1+\tau, \\
      1+\tau, \tau \right),  
  \left( 1+\tau, 1+\tau, 1, 1 \right), \\
  \left( 1+\tau, 2 \tau, \tau, 
      \tau \right),  
  \left( 1+\tau, 2 \tau, \tau, 1 \right), 
  \left( \tau, \tau, \tau, 1 \right), 
  \left( \tau, 1+\tau, \tau, 1 \right),  
  \left( 1+\tau, 1+\tau, \\
      \tau, 1 \right), \\
  \left( \tau, 2, 3-\tau, 1 \right), 
  \left( 2+3\tau, 2+4\tau , 1+3\tau
  , 1+\tau \right)\rbrace $.
\end{footnotesize}

The arrangement $A^3_1(27)$ is obtained by leaving out the last element in the above listing. We note that both arrangements have $\mathcal{A}(H_3)$ as parabolic subarrangement. Thus, their minimal field of definition is $\mathbb{Q}(\tau)=\mathbb{Q}(\sqrt{5})$ (extending the terminology introduced in \cite{p-C10b} to arrangements in $\mathbb{P}^3$).
\end{section}

\def\cprime{$'$}
\providecommand{\bysame}{\leavevmode\hbox to3em{\hrulefill}\thinspace}
\providecommand{\MR}{\relax\ifhmode\unskip\space\fi MR }
\providecommand{\MRhref}[2]{%
  \href{http://www.ams.org/mathscinet-getitem?mr=#1}{#2}
}
\providecommand{\href}[2]{#2}

\end{document}